\documentclass[12pt]{article}
\usepackage{amssymb}    % AMS TeX math symbols
\usepackage{amsmath}    % AMS Math package
\usepackage{amsthm} % AMS Theorem package
\usepackage{textcomp}
\usepackage{ccfonts}
\usepackage{psfrag}
\usepackage{epsfig}
\usepackage{color}

\usepackage{latexsym}
\usepackage{amscd}
\newtheorem{theorem}{Theorem}

\newtheorem{lemma}{Lemma}

\def\PROB {{\mathbb P}}
\def\EXP {{\mathbb E}}
\def\IND{{\mathbb I}}
\def\Var{{\mathbb Var}}

\def\R{{\mathbb R}}

\def \eL {\stackrel{{\cal L}}{=}}
\def\esssup{\mathop{\rm ess\, sup}}

\begin{document}
%\begin{titlepage}
%\date{}
\title{On the consistency of the Kozachenko-Leonenko entropy estimate}
\author{Luc Devroye\thanks{Luc Devroye, School of Computer Science, McGill University, Montreal, Canada. lucdevroye@gmail.com.
Luc's research was supported by a Discovery Grant from NSERC.}
\and
L\'aszl\'o Gy\"orfi\thanks{L\'aszl\'o Gy\"orfi, Department of Computer Science and Information Theory, Budapest University of Technology and Economics, Budapest, Hungary. gyorfi@cs.bme.hu.}}

\maketitle
\begin{abstract}
We revisit the problem of the estimation of the differential entropy
$H(f)$ of a random vector $X$ in $R^d$
with density $f$, assuming that $H(f)$ exists and is finite.
In this note, we study the consistency of the popular nearest neighbor
estimate $H_n$
of Kozachenko and Leonenko.
Without any smoothness condition we show that the estimate is consistent ($E\{|H_n - H(f)|\} \to 0$ as
$n \to \infty$)
if and only if $\EXP \{ \log ( \| X \| + 1 )\} < \infty$.
Furthermore, if $X$ has compact support, then $H_n \to H(f)$ almost surely.
\end{abstract}

{\sc Index terms}: differential entropy estimate,   consistency conditions
%\end{titlepage}

\section{Introduction}

The differential entropy of a random $\R^d$-valued vector $X$ with  probability density function  $f$ is
\begin{equation}
\label{entropy}
H(f) = - \int f(x) \ln f(x) dx =-\EXP\{f(X) \}
\end{equation}
when this integral exists.

The objective of this paper is to study an estimate of (\ref{entropy}) based on independent and identically distributed samples $X_1,\dots X_n$, with density $f$.

Estimation of the differential entropy has a long history.
The early part of that story was described by  Beirlant et al. \cite{BeDuGyvdM97}.
Since then, there has been considerable activity on the topic.
There are estimates based
on kernel methods
(Joe \cite{Joe89};
Gy\"orfi and van der Meulen \cite{GyvdM91};
Hall and Morton \cite{HaMo93};
Shwartz et al. \cite{ShZiSc05};
Paninski and Yajima \cite{PaYa08};
Krishnamurthy et al. \cite{KrKaPoWa14};
Kandasamy et al. \cite{KaKrPoWa15}),
on wavelet functions,
(Donoho et al. \cite{DoJoKePi96};
Delyon et al. \cite{DeJu96}, and
Chesneau et al. \cite{ChNaSe17}),
on partitioning methods
(Stowell et al. \cite{StPl09}),
and on nearest neighbour strategies
(Kozachenko and Leonenko \cite{KoLe87};
Tsybakov and Van der Meulen \cite{TsvM94};
Singh et al. \cite{SiMiHnFeDe03};
Sricharan, Raich, and Hero \cite{SrRaHe12};
Sricharan, Wei, and Hero \cite{SrWeHe13};
Singh and  P\'oczos \cite{SiPo16};
Gao, Oh, and Viswanath \cite{GaOhVi16} and \cite{GaOhVi17};
Delattre and Fournier \cite{DeFo18};
Lord, Sun, and Bollt \cite{LoSuBo18};
Berrett, Samworth, and Yuan \cite{BeSaYu17}).
Other related work deals with general properties of functional estimation (see, e.g.,
Birg\'e and Massart \cite{BiMa95})
or sufficient conditions of consistency of the so-called integral estimate $\int f_n \log (1/f_n)$ when
$f_n$ is a general density estimate (see, e.g.,  Godavarti \cite{God04}).

There has been particular interest in minimax rates of convergence,
culminating in the paper by Han, Jiao, Weissman and Wu \cite{HaJiWeWu20} who
obtained the minimax rates for classes of densities on the unit cube of $\R^d$ that are
Besov balls or H\"older (Lipschitz) balls of smoothness parameters $s$. Their rates
are of exact order $\max ( 1/(n \log n)^{s/(s+d)} , 1/\sqrt{n} )$. For
standard Lipschitz densities, the minimax rate is of exact order $1/\sqrt{n}$
(achieving the parametric rate) for
dimension one, but it is $1/(n \log n)^{-1/(d+1)}$ when $d > 1$.
They also construct a kernel-based  estimate that is minimax optimal for these classes.

Most of the previous work assumes compact support, and thus sidesteps
the thorny issue of infinite tails.
The objective of this note is merely to give a complete characterization of the
consistency of the popular Kozachenko-Leonenko estimate \cite{KoLe87} for estimating
the differential entropy.
Without any smoothness condition on the density $f$ we study three types of consistency: strong, $L_1$ and weak.

It is an open research problem, whether there exists an entropy estimate such that it is consistent under the only condition that $H(f)$ is finite.
The most obvious way of estimating the differential entropy is the partitioning-based estimate.
If the corresponding partition is deterministic, then we conjecture that there exist no deterministic partitions
such that the relating differential entropy estimate is a.s. consistent for any finite differential entropy.
The significant breakthrough in this respect is due to Wang, Kulkarni, Verdu \cite{WaKuVe05} and Silva, Narayanan \cite{SiNa10}.
They suggested data-dependent partitioning, for which the partitioning-based estimates of  Kullback-Leibler (KL) divergence and of mutual information are strongly consistent under the only condition that the KL-divergence and the mutual information are finite, respectively.
We guess that a universally consistent entropy estimate can be derived from data dependent partitioning.

For the  cross-validation estimate or leave-one-out entropy estimate,
$f_{n,i}$ denotes a density estimate based
on $X_1,\dots X_n$ leaving $X_i$ out, and the corresponding entropy
estimate is of the form
\begin{equation}
H_n=-\frac 1n \sum_{i=1}^n\ln f_{n,i}(X_i). \label{CVE}
\end{equation}
Kozachenko and Leonenko \cite{KoLe87} introduced the
nearest neighbor entropy estimate  as follows. Let $R_{n,i}(x)$, $x\in \R^d$, be defined by
\begin{align*}
R_{n,i}(x)=\min _{j\ne i, j\le n}\|x-X_j\|,
\end{align*}
where $\| \cdot \|$ denotes the Euclidean norm.
Then the  nearest neighbor entropy estimate is
\begin{align}
\label{NN}
H_n=\frac 1n \sum_{i=1}^{n}\ln ((n-1)R_{n,i}(X_i)^dv_d)+ C_E,
\end{align}
where  $C_E=-\int_0^{\infty} e^{-t}\ln tdt=0.5772...$ is the Euler-Mascheroni constant and $v_d$ denotes the volume of the unit sphere in $\R^d$.
The estimate in (\ref{NN}) has the form of (\ref{CVE}) if $f_{n,i}$ is a constant multiple of the first-nearest-neighbor (1-NN) density estimate:
\begin{align*}
f_{n,i}(x)=\frac{1}{(n-1)\min _{j\ne i, j\le n}\|x-X_j\|^dv_de^{C_E}}.
\end{align*}
Notice that this particular $f_{n,i}$ is not consistent in $L_1$, because
\begin{align*}
\int f_{n,i}(x)dx=\infty.
\end{align*}
Furthermore, $f_{n,i}(x)$ is unbounded at the $X_j, j\ne i$.
Also, $f_{n,i}(x)$ does not in general tend to $f(x)$ in probability, i.e., the density estimates $f_{n,i}$  are not weakly consistent.

Under some mild conditions on the density $f$, Kozachenko and Leonenko \cite{KoLe87}
proved the mean square consistency.
Biau and Devroye \cite{BiDe15} showed that for bounded $X$, if $\int f(x)\ln^2 (f(x)+1)dx<\infty$, then $H_n\to H(f)$ in probability.

For smooth densities,
Berrett,  Samworth and Yuan \cite{BeSaYu17},
Delattre  and Fournier \cite{DeFo18}
and
Tsybakov and van der Meulen \cite{TsvM94} studied the rate of convergence of $H_n$ and of its extensions to many nearest neighbors.

In this paper we show that for bounded $X$ the Kozachenko-Leonenko estimate is strongly consistent.
Furthermore, give a necessary and sufficient tail condition for $L_1$ consistency.
In addition, construct a counterexample on weak consistency for a uniform density on an unbounded set.
The proofs are presented in the last section.

\section{Consistency results}

For bounded $X$, without any smoothness condition on the density $f$ the Kozachenko-Leonenko estimate is strongly consistent:
\begin{theorem}
\label{cons}
If the support of $f$ is  bounded and
\begin{align}
\label{l1}
\int f(x)\ln (f(x)+1)dx<\infty,
\end{align}
then
\begin{align}
\label{Mas}
\lim_{n\rightarrow \infty }H_n=H(f)
\end{align}
a.s.
\end{theorem}

\bigskip

Next, we present a necessary and sufficient tail condition on the consistency in $L_1$:

\begin{theorem}
\label{weak}
Assume that $H(f)$ is finite. For any density $f$,
\begin{align}
\label{L1}
\lim_{n\rightarrow \infty }\EXP\{|H_n-H(f)|\}=0
\end{align}
if and only if
\begin{align}
\label{llog}
\EXP\{(\ln \|X\|)^+\}<\infty.
\end{align}
\end{theorem}

If $\EXP\{(\ln \|X\|)^+\}=\infty$, then maybe the expectations $\EXP\{\ln \|X\|\}$ and $\EXP\{\ln R_{n,1}(X_1)\}$ don't exist.
However, for finite $H(f)$, we show that the expectation $\EXP\{ H_n\}$ is well defined such that it is larger than $-\infty$.

As sufficient condition, (\ref{llog}) appeared in the studies of distribution and density estimates consistent in KL-divergence.
For discrete distributions concentrated to the set of positive integers,
Gy\"orfi,  P\'ali and van der Meulen \cite{GyPavdM94} proved that a distribution with finite Shannon entropy cannot be estimated consistently in KL-divergence.
It means that for any distribution estimate $p_n=(p_{n,1},p_{n,2},\dots )$ there exist a distribution $p=(p_{1},p_{2},\dots )$ with finite Shannon entropy such that  for the KL-divergence
\begin{align*}
KL(p,p_n)
=\sum_{i=1}^{\infty}p_i\ln \frac{p_i}{p_{n,i}}
=\infty
\end{align*}
for all $n$ a.s.
However, under (\ref{llog}) one can construct a distribution estimate $p_n$ such that $KL(p,p_n)\to 0$ a.s.
This positive finding has been generalized to density estimation consistent in KL-divergence.
For the condition (\ref{llog}), one can create a density $g$ of power tail such that the mixture of the ordinary histogram and of this $g$ is consistent in KL-divergence, see Barron, Gy\"orfi and van der Meulen \cite{BaGyvdM92}.

\bigskip

Theorem \ref{weak} is a complete characterization of the consistency in $L_1$.
In fact, we show that there are two cases:
\begin{itemize}
\item
either $\EXP\{(\ln \|X\|)^+\}=\infty$ and  then $\EXP\{H_n\}=\infty$, for all $n$,
\item
or $\EXP\{(\ln \|X\|)^+\}<\infty$ and then $\lim_{n\rightarrow \infty }\EXP\{|H_n-H(f)|\}=0$.
\end{itemize}
In the sequel,  we show an example, where the density is uniform on an unbounded set and $H_n\to\infty$ in probability..

Denote by $\lambda$ the Lebesgue measure and by $\mu$ the distribution of $X$.
For $d=1$, let
\begin{align*}
f=\IND_A,
\end{align*}
where $A=\cup_j A_j$ with disjoint intervals $A_j$ such that $\lambda (A)=1$.
Then, $X$ is uniformly distributed on $A$ and therefore $H(f)=0$.
For $j\ge 1$, let $\Delta_j =\frac{1}{j(j+1)}$ and $a_j=2^{2^j}$.
Set $A=\cup_{j\ge 1} [a_j,a_j+\Delta_j]$.
Then $\lambda (A)=1$. Note that in this example $\EXP\{(\ln \|X\|)^+\}=\infty$, and therefore $\EXP\{H_n\}=\infty$, for all $n$.

\begin{theorem}
\label{noncons}
In this setup, $\lim_{n\rightarrow \infty }H_n=\infty$ in probability.
\end{theorem}

\section{Proofs}

\begin{proof}[Proof of Theorem \ref{cons}]
Put
\begin{align*}
\bar f_h(x)=\frac{\mu(B(x,h))}{\lambda(B(x,h))}=\frac{\int_{B(x,h)}f(z)dz}{\lambda(B(x,h))},
\end{align*}
where $B(x,h)$ stands for the sphere centered at $x$ and having the radius $r$.
Then,
\begin{align*}
&H_n-H(f)\\
&=
\frac 1n \sum_{i=1}^{n}\ln ((n-1)\lambda(B(X_i,R_{n,i}(X_i))))+ C_E-H(f)\\
&=
-\frac 1n \sum_{i=1}^{n}\ln \frac{\mu(B(X_i,R_{n,i}(X_i)))}{\lambda(B(X_i,R_{n,i}(X_i)))}-H(f)+ \frac 1n \sum_{i=1}^{n}\ln ((n-1)\mu(B(X_i,R_{n,i}(X_i))))+C_E\\
&=
\tilde H_n-H(f)+M_n+C_E,
\end{align*}
where
\begin{align*}
\tilde H_n
&=
-\frac 1n \sum_{i=1}^{n}\ln \bar f_{R_{n,i}(X_i)}(X_i),
\end{align*}
and
\begin{align*}
M_n
&=
 \frac 1n \sum_{i=1}^{n}\ln ((n-1)\mu(B(X_i,R_{n,i}(X_i)))).
\end{align*}
Biau and Devroye \cite{BiDe15} showed that the distribution of
$M_n$ does not depend on the density $f$, and $\EXP\{M_n\}=-C_E+O(1/n)$ and $\Var (M_n)=O(1/n)$.
The problem left is to show
\begin{align}
\label{AS2}
M_n-\EXP\{M_n\}
\to 0
\end{align}
a.s. and
\begin{align}
\label{AS1}
\tilde H_n
\to H(f)
\end{align}
a.s.\\
The proof of (\ref{AS2}) relies on the following extension of the Efron-Stein inequality
for the centered higher moments:
\begin{lemma}
\label{exp}
(Devroye et al.\ \cite{DeGyLuWa18})
Let $Z=(Z_1,\ldots,Z_n)$ be a collection of independent random
variables taking values in some measurable set $A$ and denote by
$Z^{(i)}=(Z_1,\ldots,Z_{i-1},Z_{i+1},\ldots,Z_n)$ the collection with
the $i$-th random variable dropped. Let $f:A^n\to \R$ be  a measurable real-valued function and the function $g_i:A^{n-1}\to \R$ is obtained from $f$ by dropping the $i$-th argument, $i=1,\dots ,n$.
Then for any integer $q\ge 1$,
\begin{align}
&\EXP \left[(f(Z)-\EXP f(Z))^{2q} \right]
\le (cq)^q\EXP \left[ \left(\sum_{i=1}^n \left(f(Z)-g_i(Z^{(i)})\right)^2\right)^{q}\right]\nonumber\\
&\qquad + (cq)^q\EXP \left[ \left(\sum_{i=1}^n \EXP\left[\left(f(Z)-g_i(Z^{(i)})\right)^2\mid Z_1,\dots ,Z_{i-1},Z_{i+1}, \dots ,Z_n\right]\right)^{q}\right],
\label{Wal}
\end{align}
with a universal constant $c<5.1$.
\end{lemma}

For the term $M_n-\EXP\{M_n\}$,
define $M_n^{(i)}$ as $M_n$ without the $i$-th term. We apply Lemma \ref{exp} with $q=2$:
\begin{align*}
\EXP \left[\left( M_n-\EXP\{M_n\}  \right)^{4} \right]
&\le
(c2)^2\EXP \left[ \left(\sum_{i=1}^n \left(M_n-M_n^{(i)}\right)^2\right)^{2}\right]\\
&\quad + (c2)^2\EXP \left[ \left(\sum_{i=1}^n \EXP\left[\left(M_n-M_n^{(i)}\right)^2\mid X_1,\dots ,X_{i-1},X_{i+1}, \dots ,X_n\right]\right)^{2}\right]\\
&\le
2^3c^2n^2\EXP \left[  \left(M_n-M_n^{(n)}\right)^4\right]\\
&=
2^3c^2n^2\EXP \left[  \left(M_n-\frac{n-1}{n}M_{n-1}\right)^4\right].
\end{align*}
Noting that
\begin{align*}
&\left(M_n-\frac{n-1}{n}M_{n-1}\right)^4\\
&=
\frac{1}{n^4} \left(\sum_{i=1}^{n}\ln ((n-1)\mu(B(X_i,R_{n,i}(X_i))))-\sum_{i=1}^{n-1}\ln ((n-2)\mu(B(X_i,R_{n-1,i}(X_i))))\right)^4\\
&=
\frac{1}{n^4} \Big(-\ln ((n-1)\mu(B(X_n,R_{n,n}(X_n))))+(n-1)\ln \frac{ n-2}{ n-1}\\
&\quad +\sum_{i=1}^{n-1}\ln \frac{ \mu(B(X_i,R_{n-1,i}(X_i)))}{\mu(B(X_i,R_{n,i}(X_i)))}\IND_{R_{n,i}(X_i)<R_{n-1,i}(X_i) }\Big)^4,
\end{align*}
the $c_r$-inequality and  Jensen's inequality imply that
\begin{align*}
\EXP \left[\left( M_n-\EXP\{M_n\}  \right)^{4} \right]
&\le
\frac{2^3 3^3c^2}{n^2}\Big(
\EXP \left[\left( \ln ((n-1)\mu(B(X_n,R_{n,n}(X_n))))\right)^{4} \right]
+1\\
&+\EXP \left[\left(\sum_{i=1}^{n-1}\ln \frac{ \mu(B(X_i,R_{n-1,i}(X_i)))}{\mu(B(X_i,R_{n,i}(X_i)))}\IND_{R_{n,i}(X_i)<R_{n-1,i}(X_i) } \right)^{4} \right]
\Big)\\
&\le
\frac{6^3 c^2}{n^2}\Big(
\EXP \left[\left( \ln ((n-1)\mu(B(X_n,R_{n,n}(X_n))))\right)^{4} \right]
+1\\
&+\EXP \left[\left(\sum_{i=1}^{n-1}\IND_{R_{n,i}(X_i)<R_{n-1,i}(X_i) } \right)^{3}
\sum_{i=1}^{n-1}\left(\ln \frac{ \mu(B(X_i,R_{n-1,i}(X_i)))}{\mu(B(X_i,R_{n,i}(X_i)))} \right)^{4} \right]
\Big).
\end{align*}
For any $x$, $\mu(B(x,\|x-X_i\|))$ is uniformly distributed on $[0,1]$ (see Section 1.2 in \cite{BiDe15}), and therefore
\begin{align*}
(n-1)\mu(B(X_n,R_{n,n}(X_n)))
&=
(n-1)\mu(B(X_n,\min_{1\le i \le n-1}\|X_i-X_n\|))\\
&=
(n-1)\min_{1\le i \le n-1}\mu(B(X_n,\|X_i-X_n\|)).
\end{align*}
It implies, that for given $X_n$
\begin{align*}
(n-1)\mu(B(X_n,R_{n,n}(X_n)))
&\eL
(n-1)\min_{1\le i \le n-1}U_i,
\end{align*}
where $\eL$ denotes equality in distribution, and $U_1,\dots ,U_{n-1}$ are i.i.d.\ uniform on $[0,1]$.
Thus,
\begin{align*}
\EXP \left[\left( \ln ((n-1)\mu(B(X_n,R_{n,n}(X_n))))\right)^{4} \right]
&=
\EXP \left[\left( \ln \left((n-1)\min_{1\le i \le n-1}U_i\right)\right)^{4} \right]
=O(1).
\end{align*}
Lemma 20.6 in \cite{BiDe15} yields
\begin{align*}
\sum_{i=1}^{n-1}\IND_{R_{n,i}(X_i)<R_{n-1,i}(X_i) }
&=
\sum_{i=1}^{n-1}\IND_{\|X_i-X_n\|<R_{n-1,i}(X_i) }
\le \gamma_d
\end{align*}
a.s., where $\gamma_d$ is the minimal number of cones of angle $\pi/6$ that cover $\R^d$.
Thus,
\begin{align*}
\EXP \left[\left( M_n-\EXP\{M_n\}  \right)^{4} \right]
&\le
\frac{6^3 c^2}{n^2}\left(O(1)
+\gamma_d^3(n-1)\EXP \left[\left(\ln \frac{ \mu(B(X_1,R_{n-1,1}(X_1)))}{\mu(B(X_1,R_{n,1}(X_1)))} \right)^{4} \right]\right).
\end{align*}
One has
\begin{align*}
&(n-1)\EXP \left[\left(\ln \frac{ \mu(B(X_1,R_{n-1,1}(X_1)))}{\mu(B(X_1,R_{n,1}(X_1)))} \right)^{4} \right]\\
&=
(n-1)\int_0^{\infty}\PROB \left\{\left(\ln \frac{ \mu(B(X_1,R_{n-1,1}(X_1)))}{\mu(B(X_1,R_{n,1}(X_1)))} \right)^{4}\ge s \right\}ds\\
&=
(n-1)\int_0^{\infty}\EXP \left[\PROB \left\{ \frac{ \mu(B(X_1,R_{n-1,1}(X_1)))}{\mu(B(X_1,\|X_1-X_n\|))} \ge e^{s^{1/4}}\mid X_1,\dots ,X_{n-1} \right\}\right]ds\\
&=
(n-1)\int_0^{\infty}\EXP \left[\PROB \left\{ \mu(B(X_1,R_{n-1,1}(X_1)))e^{-s^{1/4}} \ge \mu(B(X_1,\|X_1-X_n\|))\mid X_1,\dots ,X_{n-1} \right\}\right]ds\\
&=
(n-1)\int_0^{\infty}\EXP \left[  \mu(B(X_1,R_{n-1,1}(X_1)))  \right]e^{-s^{1/4}}ds\\
&\le
\int_0^{\infty}e^{-s^{1/4}}ds.
\end{align*}
These limit relations imply
\begin{align*}
 \EXP \left[\left( M_n-\EXP\{M_n\}  \right)^{4}\right]   =O(1/n^2)~,
\end{align*}
which together with the Borel-Cantelli Lemma yields (\ref{AS2}).\\
In the proof of (\ref{AS1}) we apply Breiman's generalized ergodic theorem (see Lemma 27.2 in  \cite{GyKoKrWa02}):
\begin{lemma}
\label{Breiman}
Let $X_1,X_2,\dots$ be a stationary and ergodic sequence. If $F_n=F_n(\{X_i\})$, $n=1,2,\dots$ are random functions such that
\begin{align}
\label{Br1}
F_n(X_1)\to F(X_1)
\end{align}
a.s. and
\begin{align}
\label{Br2}
\EXP\{\sup_n|F_n(X_1)|\}<\infty,
\end{align}
then
\begin{align}
\label{Br3}
\frac 1n \sum_{i=1}^{n} F_n(X_i)\to \EXP\{F(X_1)\}
\end{align}
a.s.
\end{lemma}

If $X'_n(x)$ stands for the second nearest neighbor of $x$ among $X_1,\dots ,X_n$, then $R_{n,i}(X_i)=\|X'_n(X_i)-X_i\|$ and so
\begin{align*}
\bar f_{R_{n,i}(X_i)}(X_i)
&=
 \bar f_{\|X'_n(X_i)-X_i\|}(X_i)
\end{align*}
and
\begin{align*}
\tilde H_n
&=
-\frac 1n \sum_{i=1}^{n}\ln \bar f_{\|X'_n(X_i)-X_i\|}(X_i).
\end{align*}
Therefore, (\ref{AS1}) means that
\begin{align}
\label{AS0}
-\frac 1n \sum_{i=1}^{n}\ln \bar f_{\|X'_n(X_i)-X_i\|}(X_i)
\to H(f)
\end{align}
a.s. Defining
\begin{align*}
F_n(x)=-\ln \bar f_{\|X'_n(x)-x\|}(x)
\end{align*}
and
\begin{align*}
F(x)=-\ln f(x),
\end{align*}
we verify the conditions of  Lemma \ref{Breiman}.
The Lebesgue differentiation theorem (cf.  Theorem 20.18  in \cite{BiDe15}) yields that
\begin{align*}
\lim_{r\downarrow 0}\bar f_{r}(x)
= f(x)
\end{align*}
for $\lambda$-almost all $x$.
The Cover-Hart theorem (cf.  Lemma 2.2 in \cite{BiDe15}) implies
\begin{align*}
\|X'_n(x)-x\|
\to 0
\end{align*}
a.s. for $\mu$-almost all $x$.
As $\mu$ is absolutely continuous with respect to $\lambda$, these limit relations result in
\begin{align*}
\bar f_{\|X'_n(x)-x\|}(x)
\to f(x)
\end{align*}
a.s. for $\mu$-almost all $x$, from which (\ref{Br1}) follows.
Let $L$ denote an upper bound on $\|X\|$.
Introduce the Hardy-Littlewood maximal functions
\begin{align*}
f^*(x)=\sup_{h>0}\bar f_h(x)
\end{align*}
and
\begin{align*}
g^*(x)=\sup_{2L>h>0}\frac{1}{\bar f_h(x)}
\end{align*}
(\ref{l1})  implies that
\begin{align}
\label{M}
\int f(x)\ln (f^*(x)+1)dx<\infty,
\end{align}
while if, in addition,  $X$ is bounded, then
\begin{align}
\label{MM}
\int f(x)\ln (g^*(x)+1)dx<\infty,
\end{align}
see page 82 in \cite{BiDe15}.
Note that (\ref{MM}) is the only item in the proof, where the boundedness of $X$ is used.
Thus,
\begin{align*}
|\ln \bar f_{\|X'_n(x)-x\|}(x)|
&=
(\ln \bar f_{\|X'_n(x)-x\|}(x))^+ + \left(\ln \frac{1}{\bar f_{\|X'_n(x)-x\|}(x)}\right)^+\\
&\le
\ln (f^*(x)+1)+\ln (g^*(x)+1),
\end{align*}
and so (\ref{M}) and (\ref{MM}) result in
\begin{align*}
\EXP\{\sup_n|\ln \bar f_{\|X'_n(X_1)-X_1\|}(X_1)|\}
&\le
\int f(x)\ln (f^*(x)+1)dx+\int f(x)\ln (g^*(x)+1)dx<\infty,
\end{align*}
which yields (\ref{Br2}), and the conditions of Lemma \ref{Breiman} are verified.\\
Note, that similarly to the proof of (\ref{AS2}), we can show the universal strong law of the sum nearest neighbor balls: for any density $f$,
\begin{align*}
\sum_{i=1}^{n} \mu\left(B\left(X_i,\min _{j\ne i, j\le n}\|X_i-X_j\|\right)\right)
&\to 1
\end{align*}
a.s.
\end{proof}

\begin{proof}[Proof of Theorem \ref{weak}]
First we have to show that the expectation $\EXP\{ H_n\}$ and equivalently $\EXP\{\tilde H_n\}$ exist.
Jensen's inequality implies that
\begin{align*}
\EXP\{(\tilde H_n)^-\}
&=
\EXP\left\{\left( -\frac 1n \sum_{i=1}^{n}\ln \bar f_{R_{n,i}(X_i)}(X_i) \right)^-\right\}\\
&\ge
\frac 1n \sum_{i=1}^{n}\EXP\left\{\left( -\ln \bar f_{R_{n,i}(X_i)}(X_i) \right)^-\right\}\\
&=
\EXP\left\{\left( -\ln \bar f_{R_{n,1}(X_1)}(X_1) \right)^-\right\}\\
&\ge
-\int f(x)(\ln f^*(x))^+dx\\
&>
-\infty,
\end{align*}
when
\begin{align}
\label{M2}
\int f(x)\ln \frac{f^*(x)}{f(x)}dx<\infty.
\end{align}
Choose $0<L<\esssup_x f(x)$. Jensen's inequality implies
\begin{align*}
\int f(x)\ln \frac{f^*(x)}{f(x)}\IND_{f(x)> L}dx
\le
\left(\int f(x)\IND_{f(x)> L}dx\right)\ln \frac{\int f^*(x)\IND_{f(x)> L}dx}{\int f(x)\IND_{f(x)> L}dx}
<\infty,
\end{align*}
because (\ref{l1}) together with
\begin{align*}
\lambda(\{x:f(x)> L\})
<\infty
\end{align*}
yields
\begin{align*}
\int f^*(x)\IND_{f(x)> L}dx
<\infty,
\end{align*}
see Fefferman and Stein \cite{FeSt71}.
Furthermore,
\begin{align*}
\int f(x)\ln \frac{f^*(x)}{f(x)}\IND_{f(x)\le L}dx
&=
\int f(x)\ln f^*(x)\IND_{f(x)\le L}dx
-
\int f(x)\ln f(x)\IND_{f(x)\le L}dx\\
&\le
L\int \ln \max\{f^*(x),1\}dx
-
\int f(x)\ln f(x)\IND_{f(x)\le L}dx.
\end{align*}
Fefferman and Stein \cite{FeSt71} proved that
\begin{align}
\label{M4}
\lambda(\{x:f^*(x)>t\})\le \frac ct,
\end{align}
with $t>0$ such that $c$ depends only on the dimension $d$, see also (a) of Lemma 10.47 in \cite{WhZy77}.
By (\ref{M4}),
\begin{align*}
\int \ln \max\{f^*(x),1\}dx
&=
\int_0^{\infty} \lambda(\{x:\ln \max\{f^*(x),1\}>t\})dt\\
&=
\int_0^{\infty} \lambda(\{x: \max\{f^*(x),1\}>e^t\})dt\\
&=
\int_0^{\infty} \lambda(\{x: f^*(x)>e^t\})dt\\
&\le
\int_0^{\infty} \frac{c}{e^t}dt\\
&=c.
\end{align*}
Thus, (\ref{M2}) is verified and so we proved that $\EXP\{\tilde H_n\}$ exists.\\
{\sc Assume that $\EXP\{(\ln \|X\|)^+\}=\infty$.} Then, for any $x\in \R^d$, $\EXP\{(\ln \|X-x\|)^+\}=\infty$ and therefore $\EXP\{(\ln  R_{2,1}(X_1))^+\}=\EXP\{(\ln \|X_1-X_2\|)^+\}=\infty$.
Next we show that  $\EXP\{(\ln  R_{n,1}(X_1))^+\}=\EXP\{\min_{2\le i \le n}(\ln \|X_1-X_i\|)^+\}=\infty$, too:
Find $r$ such that $\mu(B(0,r))=1/2$.
Then
\begin{align*}
\EXP\left\{ \left(\ln R_{n,1}(X_1)\right)^+  \right\}
&\ge
\EXP\left\{ \left(\ln \frac{\|X_1\|}{2}\right)^+ \IND_{\|X_1\|\ge 2r} \IND_{X_2,\dots ,X_n\in B(0,r)}\right\}\\
&=
\frac{1}{2^{n-1}}\EXP\left\{ \left(\ln \frac{\|X_1\|}{2}\right)^+ \IND_{\|X_1\|\ge 2r}\right\}\\
&=
\infty.
\end{align*}
Thus,
\begin{align}
\label{Hinf}
\EXP\{\tilde H_n\}
&=
-\int\EXP\left\{ \ln \bar f_{R_{n,1}(x)}(x) \right\}f(x)dx\nonumber\\
&\ge
\int\EXP\left\{ \left(\ln \frac{\lambda(B(x,R_{n,1}(x)))}{\mu(B(x,R_{n,1}(x)))}\right)^+  \right\}f(x)dx\nonumber\\
&\ge
\int\EXP\left\{ \left(\ln \lambda(B(x,R_{n,1}(x)))\right)^+   \right\}f(x)dx\nonumber\\
&\ge
d\EXP\left\{ \left(\ln R_{n,1}(X_1)\right)^+  \right\}\nonumber\\
&=
\infty,
\end{align}
which yields the necessary part of the theorem:
\begin{align*}
\EXP\{|\tilde H_n-H(f)|\}
&=
\infty.
\end{align*}
{\sc Assume that $\EXP\{(\ln \|X\|)^+\}<\infty$.}
Notice that
\begin{align}
\label{Vinf}
\EXP\{ H_n\}
&=
\EXP\{ \ln ((n-1)R_{n,1}(X_1)^dv_d)\}+ C_E\nonumber\\
&\le
\ln (n-1)+\EXP\{ \ln (R_{2,1}(X_1)^dv_d)\}+ C_E\nonumber\\
&\le
\ln (n-1)+\EXP\{ (\ln (\|X_2-X_1\|^dv_d))^+\}+ C_E\nonumber\\
&<
\infty.
\end{align}
(\ref{Hinf}) and (\ref{Vinf}) means that $\EXP\{ H_n\}<\infty$ iff (\ref{llog}) holds.
By the proof of Theorem \ref{cons},
(\ref{L1}) is equivalent to
\begin{align}
\label{tL1}
\EXP\{|\tilde H_n-H(f)|\}
&\to 0.
\end{align}
Note that for $a\in \R$, $|a|=2a^+-a$ and thus, one has
\begin{align*}
\EXP\{|\tilde H_n-H(f)|\}
&\le
\int\EXP\left\{\left| \ln \frac{\bar f_{R_{n,1}(x)}(x)}{f(x)} \right|\right\}f(x)dx\\
&=
2\int\EXP\left\{\left( \ln \frac{\bar f_{R_{n,1}(x)}(x)}{f(x)} \right)^+\right\}f(x)dx
+\EXP\{\tilde H_n\}-H(f).
\end{align*}
We show that
\begin{align}
\label{w1}
\int\EXP\left\{\left( \ln \frac{\bar f_{R_{n,1}(x)}(x)}{f(x)} \right)^+\right\}f(x)dx
\to 0
\end{align}
and
\begin{align}
\label{w2}
\EXP\{\tilde H_n\}
=
-\int\EXP\left\{ \ln \bar f_{R_{n,1}(x)}(x) \right\}f(x)dx
\to H(f).
\end{align}
Concerning (\ref{w1}), we have a domination:
\begin{align*}
\left( \ln \frac{\bar f_{R_{n,1}(x)}(x)}{f(x)} \right)^+
&\le
\left( \ln \frac{f^*(x)}{f(x)} \right)^+
=
\ln \frac{f^*(x)}{f(x)},
\end{align*}
and therefore (\ref{M2}) and the pointwise convergence yield (\ref{w1}).\\
With respect to (\ref{w2}), note that (\ref{w1}) implies that
\begin{align*}
(H(f)-\EXP\{\tilde H_n\})^+
&\le
\int\EXP\left\{\left( \ln \frac{\bar f_{R_{n,1}(x)}(x)}{f(x)} \right)^+\right\}f(x)dx
\to 0
\end{align*}
and therefore
\begin{align*}
\liminf_n\EXP\{\tilde H_n\}
\ge H(f).
\end{align*}
So, we need to show
\begin{align}
\label{tH0}
\limsup_n\EXP\{\tilde H_n\}
\le H(f).
\end{align}
For
\begin{align*}
g^*_r(x)
=\inf_{0<h\le r}\bar f_{h}(x),
\end{align*}
we have that
\begin{align*}
\EXP\{\tilde H_n\}
&=
-\EXP\left\{\int f(x)\ln \bar f_{R_{n,1}(x)}(x) dx\right\}\\
&=
-\EXP\left\{\int \IND_{R_{n,1}(x)\le r}f(x)\ln \bar f_{R_{n,1}(x)}(x) dx\right\}
-\EXP\left\{\int \IND_{R_{n,1}(x)> r}f(x)\ln \bar f_{R_{n,1}(x)}(x) dx\right\}\\
&\le
-\int f(x)\ln g^*_r(x) dx
+\EXP\left\{\int f(x)\left(\ln \frac{\lambda(B(x,R_{n,1}(x)))}{\mu(B(x,r))}\right)^+ dx\right\}.
\end{align*}
The integrand of the last term tends to $0$ a.s. for $\mu$-almost all $x$ and the convergence is dominated by
\begin{align*}
\EXP\left\{\int f(x)\left(\ln \frac{\lambda(B(x,R_{2,1}(x)))}{\mu(B(x,r))}\right)^+ dx\right\}<\infty,
\end{align*}
which follows from
\begin{align*}
\EXP\{\tilde H_2\}
&=
\EXP\left\{\int f(x) \ln \frac{\lambda(B(x,R_{2,1}(x)))}{\mu(B(x,R_{2,1}(x)))} dx\right\}
<\infty.
\end{align*}
Thus, by the dominated convergence theorem
\begin{align*}
\EXP\left\{\int f(x)\left(\ln \frac{\lambda(B(x,R_{n,1}(x)))}{\mu(B(x,r))}\right)^+ dx\right\}
\to 0,
\end{align*}
and so
\begin{align*}
\limsup_n\EXP\{\tilde H_n\}
&\le
-\int f(x)\ln g^*_r(x) dx.
\end{align*}
Therefore,
\begin{align*}
\limsup_n\EXP\{\tilde H_n\}
&\le
-\sup_{0<r}\int f(x)\ln g^*_r(x) dx
=
-\lim_{r\downarrow 0}\int f(x)\ln g^*_r(x) dx
=H(f).
\end{align*}
Thus, (\ref{tH0}) is verified and so the proof of the sufficient part of the theorem is complete.
\end{proof}

\begin{proof}[Proof of Theorem \ref{noncons}]
The data can be represented as follows: let $Y_1,\dots ,Y_n$ be i.i.d. such that
\begin{align*}
\PROB\{Y_i=j\}=\Delta_j.
\end{align*}
Let $U_1,\dots ,U_n$ be i.i.d. uniform on $[0,1]$.
Set
\begin{align*}
X_i=a_{Y_i}+\Delta_{Y_i}U_i.
\end{align*}
We will recall two things from the theory of order statistics:
\begin{itemize}
\item[(i)]
If $U_1,\dots ,U_n$ are i.i.d. uniform on $[0,1]$, then the smallest neighbor distance $Z^*_n$ (the smallest 1-spacings) satisfies
\begin{align*}
Z^*_n\to E
\end{align*}
in distribution, where $E$ is standard exponential.
\item[(ii)]
With probability one,
\begin{align*}
\max_{1\le i\le n}Y_i\ge n^{1/2}
\end{align*}
except finitely often.\\
\proof
\begin{align*}
\PROB\left\{\max_{1\le i\le n}Y_i\le n^{1/2}\right\}
&=
(1-\PROB\{Y_i> n^{1/2}\})^n
\le e^{-n^{1/2}}.
\end{align*}
Apply Borel-Cantelli.
\item[(iii)]
If $Y^{**},Y^*$ are the largest and second largest $Y_i$'s, then
\begin{align*}
\PROB\{Y^{**}=Y^*\}
&\to 0.
\end{align*}
\proof
\begin{align*}
\PROB\{Y^{**}=Y^*\}
&\le
\binom n2 \PROB\left\{Y_1=Y_2\ge \max_{3\le i\le n}Y_i\right\}\\
&\le
n^2\sum_{i=1}^{\infty}\Delta_i^2(1-1/(i+1))^{n-2}\\
&\le
n^2\sum_{i\le n^{5/6}}\Delta_i^2e^{-(n-2)/(i+1)}
+
n^2\sum_{i>n^{5/6}}\Delta_i^2.
\end{align*}
The first term on the right hand side is less than
\begin{align*}
n^2\sum_{i\le n^{5/6}}e^{-(n-2)/(n^{5/6}+1)}
&=
e^{-n^{1/6}(1+o(1))}
\to 0,
\end{align*}
while the second term on the right hand side is
\begin{align*}
O(n^2/n^{5/6})
&=O(1/\sqrt{n})
\to 0.
\end{align*}
\end{itemize}
For the sake of simplicity, consider the negative of the estimate in (\ref{NN}) without the bias correction:
\begin{align}
\label{lNN}
\ell_n=\frac 1n \sum_{i=1}^{n}\log_2 \frac{1}{nZ_{i}},
\end{align}
where $Z_{i}=\min _{j\ne i, j\le n}|X_i-X_j|$.
We show that $\ell_n\to -\infty$ in probability, which implies the theorem.
Let $Y^{**}$ and $Y^*$ be the $Y$-values of the largest and second largest $X_i$'s.
Let $Z^{**}$ be the nearest neighbor distance for to largest $X_i$'s.
This is also the distance between the two largest $X_i$'s.
If $Y^{**}\ne Y^*$, then
\begin{align*}
Z^{**}
&\ge a_{Y^{**}}-a_{Y^{*}}-\Delta_{Y^{*}}
\ge \frac 12 a_{Y^{**}}
=
\frac 12 2^{2^{Y^{**} }}.
\end{align*}
Furthermore, all other nearest neighbor distances are $\ge$ (in distribution) the minimal distance between $n$ uniform order statistics
(by pushing the intervals of $A$ together), and this is $(1/n^2)$ times something that tends in law to a standard exponential $E$.
So, by (ii)
\begin{align*}
\ell_n
&\le
\frac 1n \log_2 \left[ \frac{2}{2^{2^{Y^{**} }} }\right]+\log_2 \left[ \frac{n}{E+o_P(1) }\right]\\
&=
-\frac{2^{Y^{**}}-1 }{n}+O_P(\log_2 n)\\
&\le
-\frac{1}{n}\IND_{Y^{**}\le \sqrt{n} }-\frac{2^{\sqrt{n}}-1 }{n}\IND_{Y^{**}> \sqrt{n} }+O_P(\log_2 n).
\end{align*}
Thus,
\begin{align*}
\PROB\left\{\ell_n>-\frac{2^{\sqrt{n}} }{2n}\right\}
&\le
\PROB\{Y^{**}=Y^*\}+ \PROB\{Y^{**}\le \sqrt{n}\}
\to 0,
\end{align*}
where we applied (iii).
\end{proof}


\begin{thebibliography}{10}

\bibitem{BaGyvdM92}
A. R.  Barron,  L.  Gy\"orfi,  E.  C.  van  der  Meulen.
Distribution estimation consistent in  total  variation  and two types of information divergence.
{\em IEEE Trans.  Information Theory}, 38:1437--1454, 1992.

\bibitem{BeDuGyvdM97}
J. Beirlant, E. J. Dudewicz, L. Gy\"orfi, E. C. van der  Meulen.
Nonparametric entropy estimation: an overview,
{\em International J. Mathematical and Statistical Sciences},
6:17--39, 1997.

\bibitem{BeSaYu17}
T. B. Berrett,   R. J. Samworth and M. Yuan.
Efficient multivariate entropy estimation via k-nearest neighbour distances.
{\em Annals of Statistics}, 47:288--318, 2019.

\bibitem{BiDe15}
G. Biau and L.  Devroye.
\emph{Lectures on the Nearest Neighbor Method.}
Springer---Verlag, Cham, 2015.

\bibitem{BiMa95}
L. Birg\'e and P. Massart.
Estimation of integral functionals of a density.
{\em The Annals of Statistics},  23:11--29, 1995.

\bibitem{ChNaSe17}
C. Chesneau, F. Navarro, O. S. Serea.
A note on the adaptive estimation of the differential entropy by wavelet methods.
{\em Commentationes Mathematicae Universitatis Carolinae},  58:87--100, 2017.

\bibitem{DeFo18}
S. Delattre and N. Fournier.
On the Kozachenko--Leonenko entropy estimator.
{\em J. Statistical Planning and Inference}, 185:69-–93, 2017.

\bibitem{DeJu96}
B. Delyon and A. Juditsky.
On minimax wavelet estimators.
{\em Applied Computational Harmonic Analysis}, 3:215--228, 1996.

\bibitem{DeGyLuWa18}
L. Devroye, L. Gy\"orfi, G. Lugosi, H. Walk.
A nearest neighbor estimate of the residual variance.
{\em Electronic Journal of Statistics}, 12:1752--1778, 2018.

\bibitem{DoJoKePi96}
D. L. Donoho, I. M. Johnstone, G. Kerkyacharian and D. Picard.
Density estimation by wavelet thresholding.
{\em Annals of Statistics}, 24:508--539, 1996.

\bibitem{FeSt71}
C. Fefferman, E.M. Stein.
Some maximal inequalities.
{\em Am. J. Math.} 93:107--115, 1971.

\bibitem{GaOhVi17}
W. Gao, S. Oh and P. Viswanath.
Density functional estimators with k-nearest neighbor bandwidths.
In {\em Proceedings of the 2017 IEEE International Symposium on Information Theory (ISIT)}, Aachen, Germany, 25–30 June 2017; pp. 1351--1355, 2017.

\bibitem{GaOhVi16}
W. Gao, S. Oh and P. Viswanath.
Breaking the bandwidth barrier: Geometrical adaptive entropy estimation.
In {\em Proceedings of the Advances in Neural Information Processing Systems 29 (NIPS 2016)}, Barcelona, Spain, 5–10 December 2016; pp. 2460--2468, 2016.

\bibitem{God04}
M. Godavarti.
Convergence of differential entropies.
{\em IEEE Transactions on Information Theory,} 50:1--6,  2004.

\bibitem{GyKoKrWa02}
L. Gy\"orfi, M. Kohler, A. Krzy\.zak and H. Walk.
{\em A Distribution-Free Theory of Nonparametric Regression},
Springer--Verlag, New York, 2002.

\bibitem{GyPavdM94}
L. Gy\"orfi, I. P\'ali and E. C. van der Meulen.
There is  no universal source code  for  infinite  alphabet.
{\em IEEE  Trans. Information Theory}, 40:267--271, 1994.

\bibitem{GyvdM91}
L. Gy\"orfi and E. C. van der Meulen.
On  the nonparametric  estimation   of   entropy   functional.
In {\em  Nonparametric Functional Estimation and Related  Topics},
Ed.  G.  Roussas, Kluwer Academic Publisher, pages 81--95, 1991.

\bibitem{HaMo93}
P. Hall and S. C. Morton.
On the estimation of entropy.
{\em Annals of the Institute of Statistical Mathematics}, 45:69--88, 1993

\bibitem{HaJiWeWu20}
Y. Han, J. Jiao, T. Weissman, Y. Wu.
Optimal rates of entropy estimation over Lipschitz balls.
{\em Ann. Statist.} 48:3228--3250, 2020.

\bibitem{Joe89}
H. Joe.
Estimation of entropy and other functionals of a multivariate density.
{\em Annals of the Institute of Statistical Mathematics},  41:683--697, 1989.

\bibitem{KaKrPoWa15}
K. Kandasamy, A. Krishnamurthy, B. P\'oczos and L. Wasserman.
Nonparametric von Mises estimators for entropies, divergences and mutual informations.
In {\em Advances in Neural Information Processing Systems}, pages 397--405, 2015

\bibitem{KoLe87}
L. F. Kozachenko and N. N. Leonenko.
Sample estimate of entropy of a random vector.
{\em Problems of Information Transmission},
23:95--101, 1987.

\bibitem{KrKaPoWa14}
A. Krishnamurthy, K. Kandasamy, B. P\'oczos and L. Wasserman.
Nonparametric estimation of R\'enyi divergence and friends.
In {\em International Conference on Machine Learning}, pages 919--927, 2014

\bibitem{LoSuBo18}
W. M. Lord, J. Sun  and E. M. Bollt.
Geometric k-nearest neighbor estimation of entropy and mutual information.
{\em Chaos: An Interdisciplinary Journal of Nonlinear Science}, 28, 033114, 2018.

\bibitem{OzUyEr08}
U. Ozertem, I. Uysal and D. Erdogmus.
Continuously differentiable sample-spacing entropy estimates.
{\em IEEE Trans. Neural Networks}, 19:1978--1984, 2008.

\bibitem{PaYa08}
L. Paninski and M. Yajima.
Undersmoothed kernel entropy estimators.
{\em IEEE Transactions on Information Theory}, 54:4384--4388, 2008.

\bibitem{ShZiSc05}
S. Shwartz, M. Zibulevsky and Y. Y. Schechner.
Fast kernel entropy estimation and optimization.
{\em Signal Process.}, 85:1045--1058, 2005.

\bibitem{SiNa10}
J. Silva and S. Narayanan.
Nonproduct data-dependent partitions for mutual information estimation: strong consistency and
applications.
\emph{IEEE Trans. Signal Processing}, 58:3497--3511, 2010.

\bibitem{SiMiHnFeDe03}
H. Singh, N. Misra. V. Hnizdo, A. Fedorowicz and E. Demchuk.
Nearest neighbor estimates of entropy.
{\em Am. J. Math. Manag. Sci.} 23:301--321, 2003.

\bibitem{SiPo16}
S. Singh and B. P\'oczos.
Finite-sample analysis of fixed-k nearest neighbor density functional estimators.
In {\em Advances in Neural Information Processing Systems}, pages 1217--1225, 2016.

\bibitem{SrRaHe12}
K. Sricharan, R. Raich and A. O. Hero.
Estimation of nonlinear functionals of densities with confidence.
{\em IEEE Transactions on Information Theory}, 58:4135--4159, 2012

\bibitem{SrWeHe13}
K. Sricharan, D. Wei and A. O. Hero.
Ensemble estimators for multivariate entropy estimation.
{\em IEEE Transactions on Information Theory}, 59:4374--4388, 2013.

\bibitem{StPl09}
D. Stowell and M. D. Plumbley.
Fast multidimensional entropy estimation by k-d partitioning.
{\em IEEE Signal Processing Letters}, 16:537--540, 2009.

\bibitem{TsvM94}
A. B. Tsybakov and E. C. van der Meulen.
Root-n consistent estimators of entropy for densities with unbounded support.
{\em Scand. J. Statist.}, 23:75--83, 1994.

\bibitem{WaKuVe05}
Q. Wang, S. R. Kulkarni and S. Verdu.
Divergence estimation of continuous distributions based on data-dependent partitions.
\emph{IEEE Trans. Information Theory}, 51:3064--3074, 2005.

\bibitem{WhZy77}
R. L. Wheeden and A. Zygmund.
\emph{Measure and Integral.}
Marcel Dekker, New York, 1977.

\end{thebibliography}
\end{document}